\documentclass[]{amsart}

\usepackage[utf8]{inputenc}
\usepackage[OT2,T1]{fontenc}
\DeclareSymbolFont{cyrletters}{OT2}{wncyr}{m}{n}
\DeclareMathSymbol{\Sha}{\mathalpha}{cyrletters}{"58}
\usepackage{graphicx}
\graphicspath{ {./images/} }

\usepackage[hyphens,spaces,obeyspaces]{url}
\usepackage[colorlinks,allcolors=blue,hyperindex,breaklinks]{hyperref}
\hypersetup{
           breaklinks=true,   
           colorlinks=true,   
           pdfusetitle=true,  
        }

\usepackage{orcidlink}
\title[A classification of even representations onto 3-adic SL(2)]{A classification of even representations onto 3-adic SL(2)}
\date{\today}
\usepackage[foot]{amsaddr}
\author[Peter Vang Uttenthal]{Peter Vang Uttenthal \orcidlink{0009-0001-0878-8213}}
\email{petervang@math.au.dk}
\address{Department of Mathematics, Aarhus Universitet, Ny Munkegade 118, 1530-421, DK-8000
Aarhus C, Denmark}

\newcommand{\Gal}{\operatorname{Gal}}

\newcommand{\Q}{\mathbb{Q}}
\newcommand{\Z}{\mathbb{Z}}
\newcommand{\F}{\mathbb{F}}

\newcommand{\Ad}{\operatorname{Ad}^0(\overline{\rho})}
\newcommand{\Adl}{\operatorname{Ad}^0(\overline{\rho}^{(\ell)})}
\newcommand{\cc}{\mathbb{C}}

\usepackage{amsthm,amsmath, mathrsfs, mathtools}
\usepackage{amsfonts}
\usepackage{amssymb}
\usepackage{fancyhdr}
\usepackage{IEEEtrantools}
\usepackage{tikz-cd}
\usepackage[english]{babel}
\usepackage[utf8]{inputenc}
\usepackage{csquotes}

\newtheorem{theorem}{Theorem}
\newtheorem{lemma}[theorem]{Lemma}
\newtheorem{definition}[theorem]{Definition}
\newtheorem{proposition}[theorem]{Proposition}
\newtheorem{corollary}[theorem]{Corollary}
\newtheorem{remark}[theorem]{Remark}

\usepackage{tabularx}
\usepackage{subcaption}
\usepackage{comment}

\begin{document}

\maketitle

\begin{abstract}
This paper gives a classification of even representations onto $\operatorname{SL}(2,\mathbb{Z}_3)$ of prime conductor.
In  addition, an explicit algorithm based on global class field theory is exhibited, computing an exhaustive series of such even representations.
\end{abstract}
\tableofcontents

\section{Introduction}
Let $\overline{\Q}$ be a fixed algebraic closure of the field of rational numbers $\Q$, and let $ \Gal(\overline{\Q}/\Q)$ be the absolute Galois group over $\Q$.  
If $F$ is any field, $\rho: \Gal(\overline{\Q}/\Q) \to \operatorname{GL}(2,F)$ a representation, and $\pi: \operatorname{GL}(2,F) \to \operatorname{PGL}(2,F)$ the natural projection map, the projective representation $\pi \circ \rho$ will be referred to as the projectivization of $\rho$. For an integer $n\geq 1$, let $\Gamma_1(n)$ be the congruence subgroup of $\operatorname{SL}(2,\Z)$ that reduces to the upper triangular unipotent subgroup in $\operatorname{SL}(2,\Z/n\Z)$. 
In \cite{auto}, a series of Maass wave forms $\pi^{(\ell)}$ 
on $\Gamma_1(\ell^k)$ for $k=1$ or $k=2$  
was constructed for a set of primes $\ell$
satisfying a finite list of conditions. 
Each Maass form $\pi^{(\ell)}$ on $\Gamma_1(\ell^k)$ in the series is algebraic of eigenvalue $\lambda=1/4$ and arises from an even complex Galois representation 
$$\overline{\rho}^{(\ell)}: \Gal(\overline{\Q}/\Q) \longrightarrow \operatorname{GL}(2,\mathbb{C})$$
of conductor $\ell^k$ and of tetrahedral or octahedral type via the Langlands-Tunnell theorem.  
If $\pi: \operatorname{GL}(2,\cc) \to \operatorname{PGL}(2,\cc)$ is the natural projection and if $K/\Q$ is the number field fixed by the kernel of 
$\pi\circ \overline{\rho}^{(\ell)}$, then $K$ is ramified exactly at one tame prime $\ell$, and 
$\Gal(K/\Q)$ is isomorphic to the symmetry group $A_4$ of a tetrahedron or the symmetry group $S_4$ of an octahedron. In the former case, the set of primes $\ell$ parametetrizing $\overline{\rho}^{(\ell)}$ will be denoted $\overline{\Lambda}^{(A_4)}$, and in the latter case, by $\overline{\Lambda}^{(S_4)}$.

This work focuses on the even representations of tetrahedral type. 
By definition, we will let $\overline{\Lambda}^{(A_4)}$ be the set of primes $\ell \equiv 1 \bmod 3$
such that if $\zeta_\ell$ is a primitive root of unity of order $\ell$, 
the class number $h_L$ of the cubic subfield $L$ of $\Q(\zeta_\ell)$ 
is even.
In  \cite{auto}, it was shown  that for each  $\ell \in \overline{\Lambda}^{(A_4)}$,
there is at least one complex projective representation
$\pi \circ \overline{\rho}^{(\ell)}$ of tetrahedral type ramified at $\ell$ and unramified everywhere else.
Since there is a group isomorphism $$A_4 \simeq \operatorname{PSL}(2,\F_3),$$
where $\F_3$ is the finite field with 3 elements, 
we are free to regard $\pi \circ \overline{\rho}^{(\ell)}$ as a projective representation onto $\operatorname{PSL}(2,\F_3)$.  
In \cite{auto}, it was shown that $\pi \circ \overline{\rho}^{(\ell)}$
can be lifted to a linear even representation
$\overline{\rho}^{(\ell)}$ onto $\operatorname{SL}(2,\F_3)$ 
of tetrahedral type ramified exactly at $\ell$ and that, conversely, any even representation  
$$
\overline{\rho}: \Gal(\overline{\Q}/\Q) \to \operatorname{SL}(2,\F_3)
$$
of  tetrahedral type ramified at one tame prime belongs to the series $\{ \overline{\rho}^{(\ell)}:  \ell \in \overline{\Lambda}^{(A_4)} \}$. 

In this paper, we adopt the perspective of working over $\F_3$ and study the problem: For which primes $\ell \in \overline{\Lambda}^{(A_4)}$ does $\overline{\rho}^{(\ell)}$ admit a lift $\rho$ onto $\operatorname{SL}(2,\Z_3)$? Here, $\Z_3$ is the ring of $3$-adic integers. 
Furthermore, we would like to identify a list of natural properties that allow us to classify all even representations of this kind.  
We will give an effective algorithm computing an exhaustive list of such representations, relying on global class field theory and Galois cohomology. 

In \cite{even1}, the residual representation $\overline{\rho}^{(349)}: \Gal(\overline{\Q}/\Q) \to \operatorname{SL}(2,\F_3)$ was shown to admit a lift onto $\operatorname{SL}(2,\Z_3)$. This paper is a  
generalization of the methods in \cite{even1} from one specific even representation to a whole series of representations. 

At this point, we will explain how to lift $\pi \circ \overline{\rho}^{(\ell)}$
to a linear representation onto $\operatorname{SL}(2,\F_3)$ of the same conductor. 
In the discussion, note that the center 
$Z(\operatorname{SL}(2,\F_3)) = \{\pm 1\} \simeq \F_3^\times$, 
while for a general prime $p$, 
$Z(\operatorname{SL}(2,\F_p)) = \{\pm 1 \} \subset 
Z(\operatorname{GL}(2,\F_p)) \simeq \F_p^\times$.
We proceed in two steps:
\begin{enumerate}
    \item We will show that the projective representation admits \emph{some} linear lift $\rho: \Gal(\overline{\Q}/\Q) \to \operatorname{SL}(2,\F_3)$. Here, we use that $K$ is only ramified at one prime. 
    \item Once we have some global lift $\rho$ available, we will twist it with 
    an appropriate idele class character 
    $$\chi: \mathbb{A}_\Q^\times \to Z(\operatorname{SL}(2,\F_3))$$
    to make $\chi \otimes \rho$ have the local properties we desire. 
\end{enumerate}

\emph{Step 1:} Using that $\Z/2\Z \simeq \operatorname{Z}(\operatorname{SL}(2,\F_3))$, 
the obstruction to lifting $\pi \circ \overline{\rho}^{(\ell)}$ to
$\operatorname{SL}(2,\F_3)$ defines a cohomology class  $\beta \in H^2(\Gal(\overline{\Q}/\Q), \Z/2\Z)$; hence $\beta$ can be identified with an element in the 2-torsion subgroup of the Brauer group $\operatorname{Br}(\Q)$ over $\Q$.
Recall \cite[Theorem 4.2]{milneCFT} the fundamental exact sequence of class field theory (due to Artin-Brauer-Hasse-Noether) 
\[
0 \to \operatorname{Br}(\Q) \to \bigoplus_{v} \operatorname{Br}(\Q_v) \stackrel{\Sigma}{\to} \Q /\Z \to 1
\]
where $v$ runs over all places of $\Q$ and $\Sigma(
\beta_v)_v := \sum_{v} \operatorname{inv}_v (\beta_v)$ where $\operatorname{inv}_v$ is the local invariant homomorphism of local class field theory \cite[Thm 2.1]{milneCFT}, which is an isomorphism for any finite place. 
Let $G_v :=\Gal(\overline{\Q}_v/\Q_v)$.
Since $\pi \circ \overline{\rho}^{(\ell)}$ is unramified at all places $v 
\neq \ell$, the restriction $\beta|_{G_v} = 0$ for all $v \neq \ell$. 
Combining this fact with the exactness of the fundamental sequence at the middle term,  
\[
0 = \sum_v \operatorname{inv}_v (\beta|_{G_v}) = \operatorname{inv}_\ell(\beta|_{G_\ell}),
\]
so $\beta|_{G_\ell} = 0$ as well. 
By exactness at the first term in the fundamental sequence, $\beta = 0$. Hence there exists a linear lifting 
$\rho: \Gal(\Q/\Q) \to \operatorname{SL}(2,\F_3)$
of $\pi \circ \overline{\rho}^{(\ell)}$.

\emph{Step 2:}
Once some lift $\rho$ is available, we may proceed 
as in the proof of a theorem of Tate \cite[Thm. 5, Part II]{Serre}. First, we specify, for each place $v$, a local lift $\rho'_v: G_v \to \operatorname{SL}(2,\F_3)$ such that 
\[
\rho'_v \equiv \rho|_{G_v} \bmod Z(\operatorname{SL}(2,\F_3)) = \{ \pm I\}.\]
Consider the place $v=\ell$: 
By Lemma \ref{efgatell}, $K = \Q(\pi \circ \overline{\rho}^{(\ell)})$ has inertial degree $f(\ell,K/\Q) = 1$ and ramification index $e(\ell, K/\Q) = 3$ so we may fix a basis such that  
\[
\rho(\sigma_\ell) 
\equiv \begin{pmatrix}
    1 &  \\ & 1 
\end{pmatrix}, \quad
\rho(\tau_\ell) 
\equiv \begin{pmatrix}
    1 & 1 \\ & 1 
\end{pmatrix} \mod Z(\operatorname{SL}(2,\F_3)).
\]
Hence, we define $\rho'_\ell$
by
$$
\rho'_\ell (\sigma_\ell) = \begin{pmatrix}
    1 & \\ & 1 
\end{pmatrix}, \quad \rho'_\ell(\tau_\ell) = \begin{pmatrix}
    1 & 1 \\ & 1
\end{pmatrix}.
$$
For places $v \not \in \{\ell,3\}$, let $\rho'_v$ be an unramified lift of 
$\pi \circ \overline{\rho}^{(\ell)}|_{G_v}$.
In particular, at $v=\infty$, for a generator $c$ of $G_\infty$, 
define
\[
\rho'_\infty(c) = \begin{pmatrix}
    1 & \\ & 1
\end{pmatrix}.
\]
At $v=3$, any lift $\rho_3'$ is allowed.

For all places $v$, \[
\chi_v := (\rho|_{G_v})^{-1} \rho'_v
\]
defines a character on $G_v$ valued in the center $Z(\operatorname{SL}(2,\F_3))$ with the property that  
\[
\rho'_v = \chi_v \otimes \rho|_{G_v}.
\]
By local class field theory, we can regard $\chi_v$ as a character on $\Q_v^\times$ with the property that 
 $\chi_v|_{\Z_v^\times} = 1$ for all $v\neq \ell, 3$. 
Therefore the product
$$
\chi := \otimes_v \chi_v : \mathbb{A}_\Q^\times \to Z(\operatorname{SL}(2,\F_3))
$$
defines an idele class character of $\Q$ such that $\chi|_{\Q_v^\times} = \chi_v$ for all $v$.
Equivalently, we can find a character \begin{equation} \label{chi}
    \chi: \Gal(\overline{\Q}/\Q) \to Z(\operatorname{SL}(2,\F_3))
\end{equation}
such that $\chi|_{G_v} = (\rho|_{G_v})^{-1} \rho'_v$ for all $v$.
Next, we \emph{define} $\overline{\rho}^{(\ell)}$ to be the linear representation
\[
\overline{\rho}^{(\ell)} := \chi \otimes \rho: \Gal(\overline{\Q}/\Q) \to \operatorname{SL}(2,\F_3).
\]
Then $\overline{\rho}^{(\ell)} = \chi \otimes \rho$ is a global linear lift
with the property that 
\[
\overline{\rho}^{(\ell)}|_{G_v} = \rho'_v
\]
for all $v$. 
This concludes the construction of the linear lift $\overline{\rho}^{(\ell)}$. 

To make the classification problem tractable, we impose the following natural conditions on the representations. Each  
$$
\rho: 
\Gal(\overline{\Q}/\Q) \longrightarrow \operatorname{SL}(2, \Z_3)
$$ in the series will be a unique lift of some $\overline{\rho}^{(\ell)}$ and will satisfy the properties listed below:
\begin{enumerate} 
\item  $\rho$ remains \emph{even}.
\label{even} 
    \item \label{fullimage} The image of $\rho$ is all of $\operatorname{SL}(2,\Z_3).$
    \item \label{solvable} $\rho$ is of tetrahedral type.  
    \item  \label{primeconductor} $\rho$ remains unramified away from $\ell$ and $3$.
\end{enumerate}
Firstly, we construct a series of such presentations $\rho^{(\ell)}$ for primes $\ell$ in a certain  subset  $ \Lambda^{(A_4)} \subseteq \overline{\Lambda}^{(A_4)}$ that we identify by explicit class field theoretic conditions; in particular, 
\[
\rho^{(\ell)}(\tau_\ell) = \begin{pmatrix}
    1 & 1 \\ & 1
\end{pmatrix}
\]
where $\tau_\ell$ is a generator of inertia at $\ell$.
Secondly, we will prove that the series is exhaustive, in the sense that any representation $\rho$ satisfying the properties (\ref{even}), (\ref{fullimage}), (\ref{solvable}) and (\ref{primeconductor}), and for which 
\[
\rho(\tau_\ell) \equiv 
\begin{pmatrix}
    1 & 1 \\ & 1
\end{pmatrix} \mod 9
\]
occurs in $\{ \rho^{(\ell)} : \ell \in \Lambda^{(A_4)}\}$.

Next, we define the properties precisely and comment on their significance. 
\begin{enumerate}
\item[Re (\ref{even})] A 2-dimensional representation of $\Gal(\overline{\Q}/\Q)$ is $\emph{even}$ if for any complex conjugation $c \in \Gal(\overline{\Q}/\Q),$ we have
$\det \rho (c) = 1.$ 
Even representations over $\Q$ may be thought of as those Galois representations defining \emph{totally real} number fields:
The eigenvalues of $\rho(c)$ are either $(1,1)$ or $(-1,-1)$, so it is immediate that the field 
$\Q(\pi \circ \rho)$ fixed by $\ker \pi \circ \rho$ is totally real. 
Furthermore, if $\chi$ is a character such that  
\[
(\rho \otimes \chi)(c) = \begin{pmatrix}
    1 &  \\ & 1
\end{pmatrix},
\]
then the field $\Q( \rho \otimes \chi )$ fixed by $\ker (\rho \otimes \chi)$ is totally real as well. 
\item[Re (\ref{fullimage})] If the  image of a Galois representation $\rho$ contains $\operatorname{SL}(2,\Z_3),$ we will say that $\rho$ has \emph{full image}.
    Let $\operatorname{Ad}^0(2,\F_3)$ be the space of traceless $2\times 2$ matrices over $\F_3$ on which $\operatorname{SL}(2,\F_3)$ acts by conjugation. 
    Since $H^2(\operatorname{SL}(2,\F_3), \operatorname{Ad}^0(2,\F_3)) =0$, there is a group monomorphism
    \[
\begin{tikzcd} 
\iota: \operatorname{SL}(2,\F_3) \arrow[r, hookrightarrow] & \operatorname{SL}(2, \Z_3) 
\end{tikzcd}
\]
cf. \cite[Lemma 1, p. 566]{even1}. In particular, any  residual representation  
    $$
\overline{\rho}:  \Gal(\overline{\Q}/\Q) \longrightarrow \operatorname{SL}(2,\F_3)
    $$
    gives rise to a \emph{banal} lift
    \[
    \rho = \iota \circ \overline{\rho}: \Gal(\overline{\Q}/\Q) \to \operatorname{SL}(2,\Z_3)
    \]
    for which $\ker \rho = \ker \overline{\rho}$.
    Since we would like  to avoid considering such banal representations, we will require the full image property. In addition, 
    representations with full image are automatically irreducible; in particular $\rho$ is not a ray class group character. Characters of ray class groups define abelian extensions, which are classified by class field theory.
\item[Re (\ref{solvable})] Let $\pi: \operatorname{GL}(2,\F_3) \to \operatorname{PGL}(2,\F_3)$ be the natural projection map.  
By  definition,  we will say that $\rho$ is of \emph{solvable type} if the projectivization $\pi\circ \overline{\rho}$ of the reduction $\overline{\rho}$ of $\rho$ modulo 3
has image isomorphic to one of the solvable groups $A_4$ or $S_4.$
If the image of  $\pi \circ \overline{\rho}$ is isomorphic to  $A_4,$ we say that $\rho$ (and $\overline{\rho}$) is of \emph{tetrahedral type}. 
Similarly, if the image of $\pi \circ \overline{\rho}$ is isomorphic to $S_4,$ we say that $\rho$ is of \emph{octahedral type}. 
\item[Re (\ref{primeconductor})] 
Galois representations attachted to automorphic representations, or those coming from geometry in the sense of Fontaine-Mazur, are known to be unramified away from a finite set of places  
$S$ containing $p=3$ and the archimedian place, denoted $\infty$.
To make the classification problem tractable, we impose the condition that $\rho$ be tamely ramified at exactly one prime number $\ell$. 
If this is the case, we will write $\rho^{(\ell)}$ for emphasis. 
\end{enumerate}

Below, an explicit algorithm is presented 
for choosing the subset of primes $\ell  \in \overline{\Lambda}^{(A_4)}$ for which $\overline{\rho}^{(\ell)}$ admits a lift to  $\operatorname{SL}(2,\Z_3)$ with the desired properties.
For any number field $F$ and any set of places $S$, we let $F^{(p)}_S$ denote the 
maximal $p$-elementary abelian extension of $F$ unramified away from the places in $S$. 
If $F/\Q$ is Galois, then $\Gal(F^{(p)}_S/F)$ becomes a vector space over $\F_p$ on which 
$\Gal(F/\Q)$ acts by conjugation. Also, for a number field $F$, we let $(F:\Q)$ be its degree (the dimension of $F$ regarded as a vector space over $\Q$), and we let $h_F$ denote its class number. 

 \begin{theorem} \label{algo}
The even representations onto $\operatorname{SL}(2,\Z_3)$ of prime conductors 
are indexed by a set of primes $\Lambda^{(A_4)}$
that can be computed by the following algorithm (straightforward to implement in magma or pari-gp):
Loop through all primes $\ell$, and identify those $\ell$ satisfying the following conditions:
    \begin{enumerate}
\item  $\ell \equiv 1 \bmod 3$ \label{1}
\item If $L$ is a subfield of $\Q(\zeta_\ell)$ with $(L:\Q)=3$, then 2 divides $h_L$. \label{2}
\item For $\ell$ satisfying (\ref{1}) and (\ref{2}), there is an $A_4$ extension $K/\Q$ of discriminant $\ell^2$ \cite[Proposition 13]{auto}.
Let $T=\{ 3, \ell \}$. Compute the decomposition of the ray class group 
$\Gal(K_T^{(3)}/K)$ as a module over $\F_3[A_4]$, and check that it contains at least one irreducible 3-dimensional representation of $A_4$.
\item \label{M} For any irreducible 3-dimensional representation of $A_4$ occurring in the decomposition of $\Gal(K_T^{(3)}/K)$ check that this representation is ramified at $\ell$.
\end{enumerate}
Define $\Lambda^{(A_4)}$ to be the set of primes $\ell$ satisfying all of the  four conditions listed above.
For each $\ell \in \Lambda^{(A_4)}$,
there exists a surjective even representation 
$$
\rho^{(\ell)}: \Gal(\overline{\Q}/\Q) \longrightarrow \operatorname{SL}(2, \Z_3)
$$
tamely ramified exactly at $\ell$.
\end{theorem}

\begin{theorem} \label{intro2}
Let 
$$
\rho: \Gal(\overline{\Q}/\Q) \to \operatorname{SL}(2,\Z_3)
$$
by an even representation unramified away from $3$ and another prime $\ell$ such that $\overline{\rho}$
is surjective and unramified at 3.
Let $K$ be the field fixed by $\ker \pi \circ \overline{\rho}$. Assume that if $M$ is an extension of $K$
unramified at all places not dividing $3$ and $\ell$ 
with 
$\Gal(M/K)\simeq \Ad$, then $M/K$ is ramified at $\ell$.
Let $\tau_\ell$ be a generator of inertia at $\ell$.
\begin{enumerate}
    \item If $\rho(\tau_\ell)$ has 1 as an eigenvalue, then $\rho$ must be the surjective representation $\rho^{(\ell)}$ constructed in Theorem \ref{algo}.
    \item If $\rho(\tau_\ell)$ does not have 1 as an eigenvalue, then $\rho(\tau_\ell)$ must have eigenvalues  
    $$ 
    \begin{pmatrix}
        \zeta_3 & \\ & \zeta_3^{-1}
    \end{pmatrix}
 $$ 
 and $\rho$ is no more ramified at $\ell$ than $\overline{\rho}$.
 In fact, $\rho$ must be the banal lift. 
\end{enumerate}
\end{theorem}

Consider an even representation 
\[
\rho: \Gal (\overline{\Q}/\Q) \to \operatorname{SL}(2,\Z_3)
\]
unramified away from $3$ and another prime $\ell$ such that $\overline{\rho}$ is irreducible and unramified at $3$.
It is a corollary of Theorem \ref{intro2} that if 
\[
\rho(\tau_\ell)^3 \not \equiv \begin{pmatrix}
    1 &  \\ & 1
\end{pmatrix} \mod 9
\]
then $\rho$ belongs to the series $\{ \rho^{(\ell)} :  \ell \in  \Lambda^{(A_4)} \}$ of surjective $3$-adic even representations.

Together, Theorems \ref{algo} and \ref{intro2} give a solution to the problem of explicitly constructing and classifying all even representations $\rho$ 
onto $\operatorname{SL}(2,\Z_3)$
of tetrahedral type which are 
unramified away from $3$ and one tame prime $\ell$
and 
such that $\rho \bmod 9$ 
is more ramified than $\overline{\rho}$ at $\ell$.

Note that the index set $\Lambda^{(A_4)}$
defines an ordering on $\{ \rho^{(\ell)} :  \ell \in  \Lambda^{(A_4)} \}$, so we may speak of the smallest element, the second smallest, etc. 
A consequence of Theorem \ref{algo} is an explicit construction of 
the smallest known \emph{surjective} $3$-adic even representation
$$
\rho^{(\ell)}: \Gal(\overline{\Q} / \Q) \longrightarrow \operatorname{SL}(2, \Z_3)
$$
of tetrahedral type measured by conductor. 
We will say that a prime $\ell$ is a Shanks prime if $\ell =a^2+3a+9$ for some integer $a \geq -1$. Shanks primes were tabulated in \cite{shanks} and form a subset of all primes satisfying $\ell \equiv 1 \bmod 3$. 
The  smallest known Shanks prime $\ell$ occurring as the parameter of some $\rho^{(\ell)}$ is $\ell  =  163$, and the smallest known prime not of Shanks type
is the prime $\ell = 277.$ The representation $\rho^{(349)}$ constructed in
\cite{even1} is the third smallest member of $\{ \rho^{(\ell)} :\ell \in \Lambda^{(A_4)} \}$.

\begin{remark}
    We conjecture that the series $\{ \rho^{(\ell)} : \ell\in \Lambda^{(A_4)} \}$ is infinite: There are infinitely many primes $\ell$ which is the only tamely ramified prime in a totally real $\operatorname{SL}(2,\Z_3)$-extension of $\Q$.
    Furthermore, we believe the set of such primes have positive density. These conjectures are inspired by the Cohen-Lenstra heuristics on the distribution of class groups:
    \begin{itemize}
        \item 
    In the tetrahedral case, by Theorem \ref{algo}, the conjecture contains the statement that there are 
    infinitely many totally real cubic fields $L^{(\ell)}$ ramified at $\ell$ with an even class number.  
    \item 
    In the octahedral case, the infinitude of $\overline{\Lambda}^{(S_4)}$ implies, in particular, that there are infinitely many totally real quadratic fields ramified at exactly one prime $\ell$ such that the class number of $\Q(\sqrt{\ell})$ is divisible by 3. More details on the octahedral case will appear in \cite[Section 4]{auto}.
    \end{itemize}
\end{remark}

\subsection{Acknowledgments}
I am grateful to Ravi Ramakrishna for inspiring me to pursue this work, 
to Frank Calegari for detailed remarks on earlier drafts that helped improve the paper, 
and to Paul Nelson for helpful conversations.
I would like to thank the referee for their insightful suggestions that resulted in a more complete manuscript.   
This work was supported by a research grant (VIL54509) from Villum Fonden.

\section{Balancing the global setting}

Let $p$ be a prime number and let $S$ be a finite set of places in $\Q$ containing $p$ and the archimedian place $\infty$. The maximal extension of $\Q$ unramified outside $S$ is denoted $\Q_S.$ 
Consider a representation
$$
\overline{\rho}: \Gal(\Q_S/\Q) \longrightarrow \operatorname{GL}(2,\F_p).
$$
Let $\operatorname{Ad}^0(\overline{\rho})$ be the
space of traceless $2\times 2$ matrices over $\F_p$ equipped with the action of $\Gal(\Q_S/\Q)$ induced by $\overline{\rho}$ and matrix conjugation. 
Let $\mu_p$ be the $p^{\text{th}}$ roots of unity and let
$\operatorname{Ad}^0(\overline{\rho})^* = \operatorname{Hom}(\operatorname{Ad}^0(\overline{\rho}),\mu_p)$.

For $v\in S$, a subspace 
$\mathcal{N}_v$ of the local cohomology group $H^1(\Gal(\overline{\Q}_v/\Q_v),\operatorname{Ad}^0(\overline{\rho}))$ is called a local Selmer condition, while
the collection $\mathcal{N}=(\mathcal{N}_v)_{v\in S}$ is referred to as a (global) Selmer condition.  
Let 
$$
\varphi: H^1 (\Gal(\Q_S/\Q), \operatorname{Ad}^0(\overline{\rho})) \to \bigoplus_{v\in S} H^1(\Gal(\overline{\Q}_v/\Q_v), \operatorname{Ad}^0(\overline{\rho}))
$$
denote the localization map, and let $\mathcal{N}=(\mathcal{N}_v)_{v\in S}$ be a Selmer condition. 
We define the Selmer group 
$H^1_{\mathcal{N}} (\Gal(\Q_S/\Q), \operatorname{Ad}^0(\overline{\rho}))
$ attached to $\mathcal{N}$
as the group
$$
H^1_{\mathcal{N}} (\Gal(\Q_S/\Q), \operatorname{Ad}^0(\overline{\rho})) := \varphi^{-1} \left(
\bigoplus_{v\in S} \mathcal{N}_v \right).
$$
For $v\in S$, let $\mathcal{N}_v^\perp
\subseteq H^1(\Gal(\overline{\Q}_v/\Q), \operatorname{Ad}^0(\overline{\rho})^*)$ be the annihilator of $\mathcal{N}_v$ under the local pairing of Galois cohomology. For the localization map 
$$
\varphi^*: H^1 (\Gal(\Q_S/\Q), \operatorname{Ad}^0(\overline{\rho}))^*) \to \bigoplus_{v\in S} H^1(\Gal(\overline{\Q}_v/\Q_v), \operatorname{Ad}^0(\overline{\rho}))^*)
$$
the dual Selmer group 
$H^1_{\mathcal{N}^\perp} (\Gal(\Q_S/\Q), \operatorname{Ad}^0(\overline{\rho})^*)$
attached to $\mathcal{N}^\perp =(\mathcal{N}_v^\perp)_{v\in S}$ is defined as the group 
$$H^1_{\mathcal{N}^\perp} (\Gal(\Q_S/\Q), \operatorname{Ad}^0(\overline{\rho})^*) :=
(\varphi^*)^{-1}\left(\bigoplus_{v\in S}\mathcal{N}_v^\perp \right).
$$
\begin{definition} \label{balancedrankn}
For $v\in S$, suppose the local versal deformation ring $R_v$ has a smooth quotient
\[
\begin{tikzcd} 
R_v \arrow[r, twoheadrightarrow] & \Z_p [[T_1,\ldots,T_{n_v}]]
\end{tikzcd}
\]
with tangent space
$\mathcal{N}_v \subseteq H^1(\Gal(\overline{\Q}_v/\Q_v), \Ad ).$ 
We say that the global setting is balanced if
$$\dim_{\F_{p}} H_{\mathcal{N}}^1(\Gal(\Q_S/\Q), \Ad)) 
    = \dim_{\F_{p}} H_{\mathcal{N}^\perp}^1(\Gal(\Q_S/\Q),\Ad^*).$$
If $n$ denotes this common $\F_p$-dimension of the Selmer and dual Selmer group, respectively, we refer to $n$ as the rank of the global setting. 
\end{definition}

In the deformation theory of Galois representations, the Selmer group can be seen as an obstruction to lifting a residual representation $\overline{\rho}$ from $\operatorname{GL}(2,\F_p)$ to $\operatorname{GL}(2,\Z_p)$. In the Galois cohomological method (\cite{artinII}, \cite{FKP1}, \cite{SFM}), the lifting problem is approached in two steps: 
\begin{enumerate}
    \item For the set $S$ of places at which the residual representation $\overline{\rho}$ is ramified, find a Selmer condtion $\mathcal{N}=(\mathcal{N}_v)_{v\in S}$ 
    for which the global setting is balanced. 
    \item If the global setting is balanced of rank $n$ for some $n\geq 1$, allow ramification at additional primes with the effect of lowering the Selmer rank, while ensuring that the global setting remains balanced. 
\end{enumerate}
The first step in the method entails finding, for each $v\in S$, a class $\mathcal{C}_v$ of local representations at $v$ that are always liftable: 

\begin{definition} Let $v\in S.$
Let $\mathcal{C}_v$ be the set of deformations of $\overline{\rho}|_{\Gal(\overline{\Q}_v/\Q_v)}$ that factor through the smooth quotient $\Z_p[[T_1,\ldots,T_{n_v}]]$, 
where $n_v=\dim \mathcal{N}_v$.
\end{definition}

With this definition of $\mathcal{C}_v$, the space $\mathcal{N}_v$ is precisely the space of local cohomology classes that preserve $\mathcal{C}_v$. 
To build a global setting that is balanced, we need to 
arrange for two conditions to be satisfied: 
\begin{enumerate}
    \item \label{balance1} For the places $v \in S$ different from the prime $p$ and the archimedian place of $\Q$, 
    we need that 
$$
\dim \mathcal{N}_v = \dim H^0(\Gal(\overline{\Q}_v/\Q_v), \Ad).
$$
In the affirmative, we say that the pair $(\mathcal{N}_v, \mathcal{C}_v)$ (or just the Selmer condition $\mathcal{N}_v$) is locally balanced at $v$.
\item \label{balance2} The Selmer condtion $\mathcal{N}_p$ and the Selmer condition $\mathcal{N}_\infty$ at infinity need to satisfy the condition 
$$
\dim \mathcal{N}_p - \dim H^0(\Gal(\overline{\Q}_p/\Q_p),\Ad) + \dim\mathcal{N}_\infty-H^0(\Gal(\overline{\Q}_\infty/\Q_\infty),\Ad) =0.
$$
If this condition is satisfied, we say that $\mathcal{N}_p$ and $\mathcal{N}_\infty$ balance each other out.
\end{enumerate}
Below, we will achieve both of these balancing conditions (\ref{balance1}) and (\ref{balance2}) in the concrete global setting of this paper. With these conditions in place, the final step is to invoke Wiles' formula \cite[Proposition 1.6]{Wiles}, which we recall in Theorem \ref{Wiles formula} below, to prove that the global setting is balanced. 

We will need one final definition which will be useful in the sequel.
\begin{definition}
Let $R_{(\mathcal{N})_{v \in S}}$  be the global deformation ring parametrizing deformations $\rho$ of $\overline{\rho}$ such that 
$\rho|_{\Gal(\overline{\Q}_v/\Q_v)} \in \mathcal{C}_v$ at all $v\in S.$
\end{definition}
The key feature guaranteeing the existence of the ring $R_{(\mathcal{N})_{v \in S}}$ is that the conditions defining $\mathcal{N}_v$ for $v\in S$ are deformation conditions in the sense of \cite[Section 23 and 25]{Mazur2}. For more details, see the proof of Lemma 2 in \cite{KR}. 

Note that $R_{(\mathcal{N})_{v \in S}}$ is a global deformation ring whose tangent space is precisely the Selmer group $H^1_\mathcal{N}(\Gal(\Q_S/\Q), \Ad)$. In particular, 
if the global setting is balanced of rank zero according to Definition \ref{balancedrankn}, 
then either 
$R_{(\mathcal{N})_{v \in S}} \simeq \Z_p$ or $R_{(\mathcal{N})_{v \in S}}\simeq \Z/p^n\Z$ for some $n\in \Z_{\geq 1}$.

\subsection{The Selmer condition at the tame prime} \label{selmeratell}
For any place $v$ in $\Q$ let 
$G_v = \Gal(\overline{\Q}_v/\Q_v).$ 
For a finite set of places $S$ in $\Q$, recall that $\Q_S$ is the maximal extension of $\Q$ unramified outside $S$ and define
$G_S = \Gal(\Q_S / \Q).$
Let 
\[
\overline{\rho}: \Gal(\overline{\Q}/\Q) \to \operatorname{SL}(2,\F_3)
\]
be an even representation of tetrahedral type. If 
$K = \Q(\pi \circ \overline{\rho})$ is the number field fixed by the kernel of the projective representation $\pi \circ \overline{\rho}$ then $\Gal(K/\Q)\simeq A_4$.
Since $\overline{\rho}$ is even, the field $K$ is totally real.
Suppose, furthermore, that $K$ is ramified at one prime 
$\ell \equiv 1 \bmod 3$ and unramified at all other places.  
Let $V_4$ be the unique normal subgroup of $A_4$ isomorphic to Klein's four-group, and let $L$ be the fixed field of $V_4$. 
Since there are no unramified extensions of $\Q$, the extension $L/\Q$ must be totally ramified at $\ell$. In particular, the
ramification index $e(\ell,L/\Q) =3$ and the inertial degree $f(\ell,L/\Q) =1$.
Since $\ell \equiv 1\bmod 3$, we note that $L$ is the unique totally real cubic subfield of 
$\Q(\zeta_\ell)$, where $\zeta_\ell$ is a primitive root of unity of order $\ell$. 

\begin{lemma} \label{efgatell} 
The ramification index $e(\ell,K/\Q)=3$ and the inertial degree $f(\ell,K/\Q) = 1$. Hence, $K/L$ is unramified and the class number $h_L$ of $L$ is even. 
\end{lemma}
\begin{proof}
Since tame ramification is cyclic, we must have $e(\ell,K/\Q) = 3$.
Next, we compute the inertial degree. We start with the prime factorization of $\ell$ as an ideal in $\Q(\zeta_\ell)$:
$$
\ell \Z[\zeta_\ell] = (1-\zeta_\ell)^{\ell-1}. 
$$
Recall that $\ell \equiv 1 \bmod 3$ and apply the norm $N_{\Q(\zeta_\ell)/L}$ on both sides to obtain the equality of ideals in the ring of integers $\mathcal{O}_L$ of $L$
$$
\ell^{(\ell-1)/3} \mathcal{O}_L=(N_{\Q(\zeta_\ell)/L}(1-\zeta_\ell))^{(\ell-1)}.
$$
Since $N_{\Q(\zeta_\ell)/L}(1-\zeta_\ell)$ is a prime element in $L$, we conclude that 
$$
\ell \mathcal{O}_L = (N(1-\zeta_\ell))^3.
$$
In particular, the prime ideal above $\ell$ in $L$ is principal; therefore it splits completely in the Hilbert class field of $L$.
Since $f(\ell,L/\Q) = 1$ and $K$ is contained in the Hilbert class field of $L$, we conclude that $f(\ell,K/\Q) = 1$.
\end{proof}
By Lemma \ref{efgatell}, 
$$
\overline{\rho}(\sigma_\ell) \equiv 
\begin{pmatrix}
    1 &  \\ & 1
\end{pmatrix},
\quad 
\overline{\rho}(\tau_\ell) \equiv 
\begin{pmatrix}
1 & 1 \\
 & 1
\end{pmatrix} \mod Z(\operatorname{SL}(2,\F_3)).
$$

Let us choose a character 
$\chi: G_\Q \to Z(\operatorname{SL}(2,\F_3))$
as in equation (\ref{chi}) of the introduction;
then  
\[
(\overline{\rho} \otimes \chi)(\sigma_\ell) = \begin{pmatrix}
    1 & \\ & 1
\end{pmatrix}, \quad 
(\overline{\rho} \otimes \chi )(\tau_\ell) = 
\begin{pmatrix}
    1 & 1 \\
    & 1 
\end{pmatrix}.
\]
Note that the adjoint representation $\operatorname{Ad}^0(\overline{\rho} \otimes \chi)$ remains unchanged compared to $\operatorname{Ad}^0(\overline{\rho})$ because 
the center $Z(\operatorname{SL}(2,\F_3))$ 
acts trivially. 
Define $\overline{\rho}^{(\ell)}$ as the twisted representation
\[
\overline{\rho}^{(\ell)} = \overline{\rho} \otimes \chi.
\]

Let $\mathcal{C}_\ell$ be the set of deformations of $\overline{\rho}^{(\ell)}|_{G_\ell}$ with coefficients in any complete Noetherian local $\Z_3$-algebra $A$
of the form
$$
 \sigma_\ell \mapsto 
\begin{pmatrix}
\sqrt{\ell} & y \\
0 & \sqrt{\ell}^{-1}
\end{pmatrix},
\quad 
 \tau_\ell \mapsto 
\begin{pmatrix}
1 & 1 \\
0 & 1
\end{pmatrix}
$$
for some $y$ in the maximal ideal $m_A$ of $A$. Note that $\ell \equiv 1 \bmod 3$ implies $\sqrt{\ell} \in 1+3\Z_3$.
\begin{lemma} 
The local cohomology groups at $\ell$ have
$$
\dim H^0(G_\ell, \operatorname{Ad}^0(\overline{\rho}^{(\ell)})) = 1, \quad 
\dim H^2(G_\ell, \operatorname{Ad}^0(\overline{\rho}^{(\ell)})) = 1,
$$ 
while 
$$
\dim H^1(G_\ell, \operatorname{Ad}^0(\overline{\rho}^{(\ell)} )) =2,  \quad 
\dim H^1_{\operatorname{unr}}(G_\ell, \operatorname{Ad}^0(\overline{\rho}^{(\ell)})) =1.$$
Furthermore,  
$H^1_{\operatorname{unr}}(G_\ell, \operatorname{Ad}^0(\overline{\rho}^{(\ell)}))$ is spanned by 
$$
g^{\operatorname{unr}}: 
\sigma_\ell \mapsto 
\begin{pmatrix}
0 & 1 \\
0 & 0
\end{pmatrix},
\quad 
 \tau_\ell \mapsto 
\begin{pmatrix}
0 & 0 \\
0 & 0
\end{pmatrix}.
$$
\end{lemma}
\begin{proof}
    This follows from local duality and the local Euler characteristic.
\end{proof}

\begin{lemma} \label{localatl} Let $R_\ell$ be the local versal deformation ring at $\ell$. 
There is a smooth quotient 
\[ 
\begin{tikzcd}
R_\ell \arrow[r, twoheadrightarrow] &  \Z_3[[T]]    
\end{tikzcd}
\]
with tangent space 
$\mathcal{N}_\ell =  
H^1_{\operatorname{unr}}(G_\ell, \operatorname{Ad}^0(\overline{\rho}^{(\ell)})).$ 
The class $\mathcal{C}_\ell$ consists precisely of those deformations 
 of $\overline{\rho}|_{G_\ell}$ that factor through this smooth quotient, and 
$$
\dim \mathcal{N}_\ell = \dim H^0(G_\ell, \operatorname{Ad}^0(\overline{\rho}^{(\ell)})).
$$
\end{lemma}

\subsection{Balancing $\mathcal{N}_3$ and $\mathcal{N}_\infty$}
Recall that $\overline{\rho}^{(\ell)} 
= \overline{\rho}\otimes \chi$ for $\chi$ defined as in (\ref{chi}) and note that
$\operatorname{Ad}^0(\overline{\rho}^{(\ell)})
\simeq \operatorname{Ad}^0(\overline{\rho})$ as Galois modules.

\begin{lemma} \label{ranksat3} 
The second local cohomology group $
H^2(G_3, \operatorname{Ad}^0(\overline{\rho}))$ is trivial.
\end{lemma}
\begin{proof}
Since $\Ad$ is an unramified representation of $G_3$, 
$ \operatorname{Ad}^0(\overline{\rho})^* = \operatorname{Hom}(\operatorname{Ad}^0(\overline{\rho}), \mu_3)$ has no nontrivial Galois invariants; hence $H^0(G_3, \operatorname{Ad}^0(\overline{\rho})^*) = 0$.
Since $H^2(G_3, \operatorname{Ad}^0(\overline{\rho}))$ is dual to $H^0(G_3, \operatorname{Ad}^0(\overline{\rho})^*)$ by local duality, it is trivial as well. 
\end{proof}
\begin{lemma} \label{3}
The local versal deformation ring at 3 is a power series ring over $\Z_3$, and we may let $\mathcal{C}_3$ be the class of all deformations 
of $\overline{\rho}^{(\ell)}|_{G_3}$ and 
$\mathcal{N}_3 =  
H^1(G_3, \operatorname{Ad}^0(\overline{\rho})).$ 
Then
$$
\dim \mathcal{N}_3 - \dim H^0(G_3, \operatorname{Ad}^0(\overline{\rho})) 
= 3. 
$$
\end{lemma}
\begin{proof}
The number of relations in the local versal deformation ring $R_3$ at 3 is equal to $\dim_{\F_3} H^2(G_3, \Ad) = 0$; therefore $R_3$ is a power series ring over $\Z_3$.
The remaining statement follows by the local Euler characteristic. 
\end{proof}

\begin{lemma} \label{infty}
The local cohomology groups at $\infty$ have the properties that 
$$
\dim H^0(G_\infty, \operatorname{Ad}^0(\overline{\rho})) = 3,
\quad 
H^2(G_\infty, \operatorname{Ad}^0(\overline{\rho})) = 0,
\quad 
H^1(G_\infty, \operatorname{Ad}^0(\overline{\rho})) = 0.
$$
In particular, 
$$
\mathcal{N}_\infty=0.
$$
\end{lemma}
\begin{proof}
Since $K$ is totally real, $\Ad|_{G_\infty}$ is trivial; hence 
$$
\dim H^0(G_\infty, \operatorname{Ad}^0(\overline{\rho})) = 3.
$$
The groups
$ H^i(G_\infty, \operatorname{Ad}^0(\overline{\rho}))$ for 
$i = 1, 2$ are trivial because 
$G_\infty$ has order 2 and 
$\operatorname{Ad}^0(\overline{\rho})$ has order $3^3$.
\end{proof}

\begin{lemma} \label{3andinfty}
The local Selmer conditions at $3$ and $\infty$ balance each other out in the sense that 
$$
\dim \mathcal{N}_3 - \dim H^0(G_3, \operatorname{Ad}^0(\overline{\rho}))
+ \dim \mathcal{N}_\infty - \dim H^0(G_\infty, \operatorname{Ad}^0(\overline{\rho})) = 0.
$$
\end{lemma}
\begin{proof}
This follows from Lemma \ref{3} and Lemma \ref{infty}.
\end{proof}

Next, Wiles' formula is recalled; it appears in \cite[Prop. 1.6]{Wiles}.
\begin{proposition}[Wiles' formula] \label{Wiles formula}
    Let $S$ be a finite set of places containing $p$ and the archimedian place. 
    Let $M$ be a torsion Galois module.
    For $v\in S,$ let $\mathcal{L}_v$ be a subspace of $H^1(\Gal(\overline{\Q}_v/\Q), M)$ with annihilator $\mathcal{L}_v^\perp \subseteq H^1(\Gal(\overline{\Q}_v/\Q), M^*).$ 
Then
\begin{IEEEeqnarray*}{lCl}
 & & \dim H^1_\mathcal{L} (\Gal(\Q_S/\Q), M) - \dim H^1_{\mathcal{L}^\perp} (\Gal(\Q_S/\Q), M^*)\\
 & = & \dim H^0 (\Gal(\Q_S/\Q), M) - \dim H^0 (\Gal(\Q_S/\Q), M^*) \\
 & + &  \sum_{v\in S} \dim \mathcal{L}_v - \dim H^0(\Gal(\overline{\Q}_v/\Q_v), M).  
\end{IEEEeqnarray*}
\end{proposition}

\begin{lemma}\label{balanced}
Let $S=\{\ell, 3, \infty\}$.
The global setting is balanced in the sense that 
$$
\dim H^1_{\mathcal{N}} (G_S, \operatorname{Ad}^0(\overline{\rho}))
=\dim H^1_{\mathcal{N}^\perp} (G_S, \operatorname{Ad}^0(\overline{\rho})^*).$$
\end{lemma}
\begin{proof}
Since $\overline{\rho}^{(\ell)}$ is irreducible and has image containing $\operatorname{SL}(2, \F_3)$, 
$$
\dim H^0 (G_S, \operatorname{Ad}^0(\overline{\rho})) =0, 
\quad \dim H^0 (G_S, \operatorname{Ad}^0(\overline{\rho})^*) =0.
$$
By Lemma \ref{3andinfty}, 
$$
\dim \mathcal{N}_3 - \dim H^0(G_3, \operatorname{Ad}^0(\overline{\rho}))
+ \dim \mathcal{N}_\infty - \dim H^0(G_\infty, \operatorname{Ad}^0(\overline{\rho})) = 0.
$$
By Lemma \ref{localatl},
$$
\dim \mathcal{N}_\ell - \dim H^0(G_\ell, \operatorname{Ad}(\overline{\rho}))= 0.
$$ 
By Wiles' formula (Proposition \ref{Wiles formula}), 
\begin{IEEEeqnarray*}{ll}
 & \dim H^1_\mathcal{N} (G_S, \operatorname{Ad}^0(\overline{\rho})) 
 - \dim H^1_{\mathcal{N}^\perp} (G_S, \operatorname{Ad}^0(\overline{\rho})^*) \\
= \quad & 
\dim H^0 (G_S, \operatorname{Ad}^0(\overline{\rho})) 
- \dim H^0 (G_S, \operatorname{Ad}^0(\overline{\rho})^*)
+
\sum_{q\in S} \dim \mathcal{N}_q - \dim H^0(G_q, \operatorname{Ad}^0(\overline{\rho})) \\
= \quad & 0.
\end{IEEEeqnarray*}
\end{proof}

\section{Classification}

\subsection{Preliminary lemmas}

\begin{lemma}[Classification of absolutely irreducible $A_4$-representations] \label{a4reps}
Up to equivalence, the absolutely irreducible $A_4$-representations over the field $\F_3$ are: 
\begin{enumerate}
    \item The trivial representation, $\F_3$
\item The standard 3-dimensional representation. 
If $\overline{\rho}$ is an irreducible representation  $\Gal(\overline{\Q}/\Q) \to \operatorname{SL}(2,\F_3)$, then the standard representation is isomorphic to the adjoint representation $\Ad$.
\end{enumerate}
\end{lemma}
\begin{proof}
    See \cite{repA4}.
\end{proof}
Note that if $Y$ is any $\F_3[A_4]$ module, and if $\operatorname{Ad}^0$ occurs in the Jordan-H\"{o}lder sequence of $Y$, then the sequence splits and $\operatorname{Ad}^0$ is a direct summand \cite[Lemma 3]{even1}.

\begin{lemma} \label{globalH^1}
Let $\ell \equiv  1  \bmod 3$ be a prime and define $S=\{3, \ell, \infty \}$ and $T=\{3,\ell\}$. 
Let $\overline{\rho}: G_S \to \operatorname{SL}(2,\F_3)$ be a surjective even representation of tetrahedral type 
unramified away from $\ell$.
If $h\in H^1(G_T, \Ad)$ is ramified at $\ell$, then
\[
\dim H^1(G_T, \Ad) - \dim H^1_\mathcal{N}(G_S, \Ad) = 1.
\]
\end{lemma}

\begin{proof}
Since $\overline{\rho}$ surjects onto $\operatorname{SL}(2,\F_3)$ it must be irreducible. 
In particular, there are no global Galois invariants:
$$H^0(G_S, \Ad) = H^0(G_S, \Ad^*) = 0.$$ 
Take $\mathcal{N}_\ell = H^1_{\operatorname{unr}}(G_\ell,\Ad)$, $\mathcal{N}_\infty = 0$
and $\mathcal{N}_3 = H^1(G_3,\Ad)$. Lemma \ref{balanced} showed that  
\begin{IEEEeqnarray*}{rCl}
\dim H^1_\mathcal{N}(G_S,\Ad) - \dim H^1_{\mathcal{N}^\perp}(G_S, \Ad^*) &=&0. 
\end{IEEEeqnarray*}
Take $\mathcal{L}_\ell = H^1(G_\ell,\Ad)$
with annihilator $\mathcal{L}_\ell^\perp = 0$,
$\mathcal{L}_\infty = 0$, $\mathcal{L}_3 = H^1(G_3,\Ad)$
and use Wiles' formula to obtain
\begin{align*}
\dim H^1_{\mathcal{L}}(G_S,\Ad) - \dim H^1_{\mathcal{L}^\perp}(G_S, \Ad^*) = 1.
\end{align*}
Note that 
$
\dim H^1_{\mathcal{L}^\perp}(G_S, \Ad^*) \leq  \dim H^1_{\mathcal{N}^\perp}(G_S, \Ad^*)
$
and that $$H^1_\mathcal{L}( G_S, \Ad) \simeq H^1(G_T,\Ad).$$ 
Hence 
\begin{IEEEeqnarray*}{rCl}
0 & \leq &
\dim H^1(G_T,\Ad) - \dim H^1_{\mathcal{N}}(G_S,\Ad) \\  
& = &  1 - (\dim H^1_{\mathcal{N}^\perp}(G_S, \Ad^*) - \dim H^1_{\mathcal{L}^\perp}(G_S, \Ad^*)  ) \leq 1.
\end{IEEEeqnarray*}
We conclude that if $h \in H^1(G_T,\Ad)$ is ramified at $\ell$ then 
\[
\dim H^1(G_T, \Ad) - \dim H^1_\mathcal{N}(G_S, \Ad) = 1.
\]
\end{proof}

\subsection{Existence of the series of even representations}

\begin{theorem} \label{goodell}
    Let $\ell \equiv 1 \mod 3$ be a prime with the following properties:
    \begin{enumerate}
        \item \label{h_L} Let $L$ be the unique cubic totally real subfield 
        of $\Q(\zeta_\ell)$. 
        The class number $h_L$ of $L$ is divisible by 2. 
        In particular,
        there is an $A_4$-extension $K/\Q$ and
        $\overline{\rho}^{(\ell)}$ such that $K=\Q(\pi  \circ \overline{\rho}^{(\ell)})$.
        \item  \label{ontonew} 
        Let $T=\{3, \ell\}$, and let $K_T^{(3)}$ be the maximal 3-elementary abelian extension of $K$ unramified away from the places in $K$ above $T$. In the decomposition
        $$\Gal(K_T^{(3)}/K) 
    \simeq \operatorname{Ad}^0(\overline{\rho}^{(\ell)})^m \oplus \F_3^n
        $$
        as $\F_3[A_4]$-modules, there is at least one adjoint representation, i.e. 
        $m \geq 1.$
        \item \label{dependonellnew}
        If $M$ is an extensions of $K$ with 
        $$\Gal(M/K)\simeq \operatorname{Ad}^0(\overline{\rho}^{(\ell)})$$ 
        as $\Gal(K/\Q)$-modules then $M/K$ is ramified at $\ell$. 
     \end{enumerate}
Letting $S= \{ \infty, 3, \ell \},$
the global setting is balanced of rank zero, i.e. 
$$
\dim H^1_\mathcal{N}(G_S, 
\Adl
) = \dim H^1_{\mathcal{N}^\perp}(G_S, \Adl^* ) = 0,
$$
and
$$R_{(\mathcal{N})_{v\in S}} \simeq \Z_3.$$
If $\rho^{(\ell)}$ 
denotes the universal deformation of $\overline{\rho}^{(\ell)}$ for which 
$\rho^{(\ell)}|_{G_v}\in \mathcal{C}_v$ for all $v\in S,$ then  
$$\rho^{(\ell)}: G_S \to \operatorname{SL}(2, \Z_3)$$ 
is surjective. 
\end{theorem}
\begin{proof}
In condition (\ref{ontonew}),
Lemma \ref{a4reps} is applied to obtain a decomposition  
$$\Gal(K_T^{(3)}/K) 
    \simeq \operatorname{Ad}^0(\overline{\rho}^{(\ell)})^m \oplus \F_3^n
        $$
as $\F_3[A_4]$-modules for some integers $m,n \geq 0.$
We know from Lemma \ref{balanced} that the global setting is balanced (according to Definition \ref{balancedrankn}). 
By Lemma \ref{localatl}, the Selmer condition 
$\mathcal{N}_\ell$ equals the unramified local cohomology classes $H^1_{\operatorname{unr}}(G_\ell, \Ad).$ 
Condition (\ref{dependonellnew}) is equivalent 
to the statement any  global cohomology class  $ f\in H^1(G_S, \Ad)$ is ramfied at $\ell,$ i.e.
$f|_{G_\ell} \not \in \mathcal{N}_\ell.$
We conclude that the global setting is balanced of rank zero.
To construct the lift $\rho^{(\ell)}$, we will show by induction 
that for all integers $n \geq 1$ there exists a deformation $\rho_n : G_S\to \operatorname{SL}(2,\Z/3^n\Z)$
such that $\rho_n|_{G_v} \in \mathcal{C}_v$ for all $v\in S$.
Note that when $n=1$, this statement is a fact since 
$\rho_1 = \overline{\rho}$. 
Suppose $\rho_n: G_S\to \operatorname{SL}(2,\Z/3^n\Z)$ is a deformation of $\overline{\rho}$ such that $\rho_n|_{G_v} \in \mathcal{C}_v$ for all $v\in S$.
The obstruction to lifting $\rho_n$ to $\Z/3^{n+1}\Z$ defines a global cohomology class in $H^2(G_S,\Ad)$ whose restriction to a local cohomology class in $H^2(G_v,\Ad)$ vanishes, since $\rho_n|_{G_v} \in \mathcal{C}_v$ is always liftable for all $v\in S$. Since the dual Selmer group is trivial, the Poitou-Tate exact sequence implies that 
the linear restriction map
$$
H^2(G_S, \Ad) \to \bigoplus_{v\in S} H^2(G_v, \Ad)
$$
is injective. Thus, the global obstruction vanishes as well, and there exists a deformation $\rho_{n+1}: G_S\to \operatorname{SL}(2,\Z/3^{n+1}\Z)$ of $\rho_n$. 
For all $v\in S$, choose $f_v \in H^1(G_v,\Ad)$ such that 
$(I+3^n f_v)\rho_{n+1}|_{G_v} \in \mathcal{C}_v$. 
The Poitou-Tate exact sequence implies that 
the restriction map
$$
H^1(G_S, \Ad) \to \bigoplus_{v\in S}  H^1(G_v, \Ad) / \mathcal{N}_v
$$
is surjective, so there exists $f\in H^1(G_S,\Ad)$ such that 
$f|_{G_v} - f_v \in \mathcal{N}_v$ for all $v\in S$. 
In particular, $(I+3^n f)\rho_{n+1}$ is a global deformation of $\rho_n$
that belongs to $\mathcal{C}_v$ when restricted to $G_v$ for all $v\in S$. This completes the inductive step. 
In summary, the sequence $(\rho_n)_{n \geq 1}$ forms a projective system since
$\rho_n \equiv \rho_{n-1}$ for all $n \geq 2$; let $\rho^{(\ell)} := \lim \rho_n$ be the projective limit. The fact that $\rho_n|_{G_v} \in \mathcal{C}_v$ for all $n$ implies that
$\rho^{(\ell)}|_{G_v} \in \mathcal{C}_v$ as well whenever $v\in S$.
It remains to prove that $\rho^{(\ell)}$ surjects onto $\operatorname{SL}(2,\Z_3)$.
Let $N$ be the field fixed by the kernel of $\overline{\rho}^{(\ell)}$.
As noted in the introduction, the quadratic, totally real extension $N/K$ exists by the fundamental exact sequence of class field theory, using that $K$ is ramified at exactly one prime.
To prove surjectivity, it suffices to show that $\rho^{(\ell)}$ maps
$\Gal(\Q_S /N)$ onto the principal congruence subgroup $\Gamma = \{ g\in \operatorname{SL}(2,\Z_3): g \equiv I \bmod 3\}$ of $\operatorname{SL}(2, \Z_3).$ 
For any ring $R$, let $\operatorname{Ad}^0(2,R)$ be the additive group of $2\times 2$ matrices with coefficients in $R$ of trace zero.
The group $\Gamma$ is a pro-$3$ group with 3 topological generators, as can be seen by noting that there is an isomorphism
$$\Gamma 
\simeq \lim_{\xleftarrow{}} 
\left( \ker (
\operatorname{SL}(2,\Z/3^{n}\Z) \to 
\operatorname{SL}(2,\Z/3\Z)) 
\right)
$$
of topological groups 
and defining
\[
\Gamma_n := \left( \ker (
\operatorname{SL}(2,\Z/3^{n}\Z) \to 
\operatorname{SL}(2,\Z/3\Z)) 
\right),
\]
there are short exact sequences 
\[1 \to 
\ker \left(
\operatorname{SL}(2,\Z/3^{n+1}\Z) \to 
\operatorname{SL}(2,\Z/3^n \Z) \right)
\to 
\Gamma_{n+1}
\to 
\Gamma_n 
\to 1
\]
for every $n \geq 1$. 
The statement then follows by induction, after noting that
$$
\ker \left(
\operatorname{SL}(2,\Z/3^{n+1}\Z) \to 
\operatorname{SL}(2,\Z/3^n \Z) \right)
= I+3^n \operatorname{Ad}^0(2,\Z/3^{n+1}\Z)
\simeq  \operatorname{Ad}^0(2,\F_3)
$$
is a vector space over $\F_3$ with 
$\dim_{\F_3} \operatorname{Ad}^0(2,\F_3) = 3$.

By the Burnside basis theorem, it suffices to prove that $\rho_2 := (\rho^{(\ell)} \bmod \ 9)$ maps $\Gal(\Q_S /N)$  onto the 3-Frattini quotient $\operatorname{Fr}(\Gamma)$.
By definition, $\operatorname{Fr}(\Gamma)$ is the maximal $3$-elementary abelian quotient of $\Gamma$, so 
the homomorphism obtained by reducing coefficents modulo 9 induces an isomorphism
$$
\operatorname{Fr}(\Gamma) \simeq \{ g\in \operatorname{SL}(2,\Z/9\Z): g \equiv I \bmod 3 \}.
$$
We see that $\operatorname{Fr}(\Gamma)$ is a vector space over $\F_3$ endowed with a $\Gal(K/\Q)$-action 
such that $\operatorname{Fr}(\Gamma)\simeq \Ad$.
Moreover, letting $\Q(\rho_2)$ be the fixed field of $\ker \rho_2$, the homomorphism
$\rho_2$ maps $\Gal(\Q(\rho_2)/N)$ onto a $\Gal(K/\Q)$-submodule of $\operatorname{Fr}(\Gamma)$.
Earlier in the proof, when $\rho^{(\ell)}$ was constructed, 
we showed that $\rho_2|_{G_\ell} \in \mathcal{C}_\ell$, which implies that
$\rho_2(\tau_\ell)$ has order $9$ in $\operatorname{SL}(2,\Z/9\Z)$.
The order of $\rho_2(\tau_\ell)$ is equal to the ramification index at $\ell$ in $\Q(\rho_2)/\Q$; since $\ell$ has ramification index $3$ in $N/\Q$, we conclude that $\Gal(\Q(\rho_2)/N)$ must be nontrivial. 
When $\rho_2$ is regarded as a map on 
$\Gal(\Q(\rho_2)/N)$, Schur's Lemma implies that the image must be all of $\operatorname{Fr}(\Gamma)$.
\end{proof}

In \cite{even1}, 
the first example of a surjective 
$p$-adic even representation
\[ 
\begin{tikzcd}
\rho: \Gal(\overline{\Q}/\Q) \arrow[r, twoheadrightarrow] &  \operatorname{SL}(2,\Z_p)   
\end{tikzcd}
\]
was given. In the present notation, the $\rho$ constructed was $\rho = \rho^{(349)}$.
At the time, this example was also the smallest one known, in the sense that 
$\operatorname{SL}(2,\Z_3)$ was realized as the Galois group of a normal extension of $\Q$ ramified only at 3 and at $\ell =349$,
where $349$ was the smallest such prime.
The theorem below relies on the availability of pari-gp, which is capable of performing fast ray class group computations over number fields of large conductors.

\begin{corollary}
Measured by conductor, 
$\rho^{(163)}$ and $\rho^{(277)}$
are the smallest even representations onto $\operatorname{SL}(2,\Z_3)$ of tetrahedral type (even if we allow the conductor to be divisible by more than one prime).  
The prime $\ell=163$ is a Shanks prime with $\ell \equiv 1 \mod 9$, while $\ell=277$ is not Shanks and has
$\ell \not \equiv 1 \mod 9$.
\end{corollary}

\begin{proof}
The existence follows from Theorem \ref{goodell}. 
By \cite[Theorem 20]{auto}, these are smallest even representations onto $\operatorname{SL}(2,\Z_3)$ of tetrahedral type. 
\end{proof}

\begin{corollary}
For primes $\ell \leq 2000,$
the group $\operatorname{SL}(2,\Z_3)$ occurs as the Galois group of distinct, totally real, normal extensions $K^{(\ell)}/\Q$ ramified at $\ell$ and at $3$ for all
$\ell \in \{163, 277, 349, 547, 607, 937, 1399, 1699, 1777, 1879, 1951\}.$
\end{corollary}

\[
\begin{tikzcd}[arrows=dash]
& \overline{\Q} & \\
K^{(163)} \arrow{ur} & K^{(277)} \arrow{u} & K^{(349)} \arrow{ul} \\
& & \\
&  \mathbb{Q} \arrow[uul, "\operatorname{SL}_2(\Z_3)"] \arrow[uu, "\operatorname{SL}_2(\Z_3)" description] \arrow[uur, "\operatorname{SL}_2(\Z_3)" swap] & 
\end{tikzcd}
\]

\begin{remark}
    Theorem \ref{goodell} contains a list of splitting conditions on $\ell$ in the extension $K_T^{(3)}$ that depends on the prime $\ell$.  
   Are there infinitely many primes $\ell$ satisfying the conditions of Theorem \ref{goodell}?
   Numerical computations suggest an affirmative answer.
\end{remark}

\subsection{Proof that the series is exhaustive} Let $p$ be a prime and let $\rho: G_\Q \to \operatorname{SL}(2,\Z_p)$ be a representation with residual representation $\overline{\rho}: G_\Q \to \operatorname{SL}(2,\F_p)$.
Let $\ell \neq p$ be a prime and let $\tau_\ell$ be a generator of tame inertia. Note that $\overline{\rho}(\tau_\ell)$ must have finite order, while 
$\rho(\tau_\ell)$ could have finite or infinite order. 
If the order of $\rho(\tau_\ell)$
is higher than the order of 
$\overline{\rho}(\tau_\ell)$, we say that  $\rho$ is  more ramified at $\ell$ than $\overline{\rho}$.

\begin{theorem} \label{exhaustion}
Let $\ell \neq 3$ be a prime. Let 
$$
\rho: \Gal(\overline{\Q}/\Q) \to \operatorname{SL}(2,\Z_3)
$$
by an even representation unramified away from $\ell$ and $3$ such that $\overline{\rho}$
is surjective and unramified at 3.
Let $K$ be the fixed field of $\ker \pi \circ \overline{\rho}$. Assume that if $M$ is an extension of $K$ unramified at all places not dividing $3$ and $\ell$
with 
$\Gal(M/K)\simeq \Ad$, then $M/K$ is ramified at $\ell$.
\begin{enumerate}
    \item If $\rho(\tau_\ell)$ has 1 as an eigenvalue, then $\rho$ must be the surjective representation $\rho^{(\ell)}$. 
    \item If $\rho(\tau_\ell)$ does not have 1 as an eigenvalue, then $\rho(\tau_\ell)$ must have eigenvalues  
    $$ 
    \begin{pmatrix}
        \zeta_3 & \\ & \zeta_3^{-1}
    \end{pmatrix}
 $$ 
 and $\rho$ is no more ramified at $\ell$ than $\overline{\rho}$; in fact, $\rho$ must be the banal lift. 
\end{enumerate}
\end{theorem}

\begin{proof}
Let $\overline{\rho}: \Gal(\overline{\Q}/\Q) \to \operatorname{SL}(2, \F_3)$ be the residual representation obtained by reducing $\rho$ modulo 3. 
Let $\pi: \operatorname{SL}(2,\F_3) \to \operatorname{PSL}(2,\F_3)$ be the canonical projection and let $K$  be  the field fixed by the kernel of $\pi \circ  \overline{\rho}$.  
First, we show that $\ell\in \overline{\Lambda}^{(A_4)}$ and that up to a twist by a character unramified away from $\ell$, we have $\overline{\rho} = \overline{\rho}^{(\ell)}$. Since $\overline{\rho}$ is assumed to be surjective, $\Gal(K/\Q) \simeq \operatorname{PSL}(2,\F_3) \simeq A_4$ (in particular, 
$\overline{\rho}$ and $\rho$ are of tetrahedral type). 
Let $V_4$ be the unique normal subgroup of $A_4$ isomorphic to Klein's four group, and let $L$ be the field fixed by $V_4$. Since $\overline{\rho}$ is assumed to be even and unramified at 3, the field
$L$  is a totally real cubic Galois extension of $\Q$ ramified at $\ell$ and unramified at all other places. 
Consequently, we must have
$\ell \equiv 1 \bmod 3$ and
$L$ is the totally real cubic subfield of $\Q(\zeta_\ell)$. 
Since the ramification at $\ell$ in $K/\Q$ is tame, $e(\ell,K/\Q) = 3$. In particular, $K/L$ is unramified, and the class number of $L$ is divisible by 4.
We conclude that $\ell$ satisfies all the conditions 
defining the set of primes $\overline{\Lambda}^{(A_4)}$.
By definition, 
$\pi \circ \overline{\rho}^{(\ell)}$ is the
projective representation obtained 
by viewing the natural homomorphism 
$$\Gal(\overline{\Q}/\Q) \to \Gal(K/\Q)$$
as a projective Galois representation onto $\operatorname{PSL}(2,\F_3)$; hence $\pi \circ \overline{\rho} = \pi \circ \overline{\rho}^{(\ell)}$.
By assumption, $\overline{\rho}$ is unramified away from $\ell$. Hence, up to a possible twist by a character unramified away from $\ell$, $\overline{\rho} = \overline{\rho}^{(\ell)}$. 

Suppose $\rho(\tau)$ has 1 as an eigenvalue. 
We claim that then $\rho = \rho^{(\ell)}$. 
Suppose, to reach a contradiction, that 
$\rho \neq \rho^{(\ell)}$. Then we may choose an integer $n \geq 1$ such that 
$
\rho_{n+1} \neq \rho_{n}^{(\ell)}
$
while
$\rho_n = \rho_{n}^{(\ell)}$. 
Recall that group $H^1(G_S,\Ad)$ acts transitively on the deformations of $\rho_n^{(\ell)}$ to $\operatorname{SL}(2,\Z/3^{n+1}\Z)$. Hence, for $S = \{3, \ell,\infty \}$, there exists some
$h \in H^1(G_S,\Ad) \setminus \{ 0 \}$ such that 
$$
(I + 3^{n} h) \rho_{n+1} = \rho_{n+1}^{(\ell)}.
$$
Let $a,b,c \in \F_3$ be such that when $h$ is  evaluated at $\tau_\ell$, 
$$
h(\tau_\ell) = \begin{pmatrix}
    a & b \\ c & -a
\end{pmatrix},
$$
and note that 
\begin{IEEEeqnarray*}{rCl}
    (I + 3^{n} h) \rho_{n+1}(\tau_\ell)
&=&\rho_{n+1}(\tau_\ell)+ 3^{n} h(\tau_\ell) \overline{\rho}(\tau_\ell) \\
&=& \rho_{n+1}(\tau_\ell)+ 3^{n}
\begin{pmatrix}
a & b \\ c & -a    
\end{pmatrix}
\begin{pmatrix}
    1 & 1\\
    0 & 1
\end{pmatrix} \\
&=&
\rho_{n+1}(\tau_\ell) 
+3^n 
\begin{pmatrix}
    a & a+b \\ c & c-a 
\end{pmatrix}
\end{IEEEeqnarray*}
Hence
$$
2 = \operatorname{Trace} \rho_{n+1}^{(\ell)}(\tau_\ell) 
= \operatorname{Trace} \rho_{n+1}(\tau_\ell) + 3^n c. 
$$
If $\rho(\tau_\ell)$ has 1 as an eigenvalue, then $\operatorname{Trace}\rho(\tau_\ell) = 2$. This forces $c=0$. Hence 
$h|_{G_\ell} \in \mathcal{N}_\ell$,  
so in fact
$h \in H^1_{\mathcal{N}}(G_S,\Ad) = 0$, a contradiction. We conclude that if $\rho(\tau_\ell)$ has 1 as an eigenvalue, then $\rho = \rho^{(\ell)}$. This is the first case in the statement of the theorem. 

The only other alternative is that $\rho(\tau_\ell)$ does not have 1 as an eigenvalue. After conjugating with an element from 
$\operatorname{GL}(2, \overline{\Q}_3)$, 
$$
\rho(\tau_\ell) = 
\begin{pmatrix}
    t & \\ & t^{-1}
\end{pmatrix}
$$
for some $t\in \overline{\Q}_3^\times$, $t\neq 1$. 
Since $\sigma_\ell$ normalizes $\tau_\ell$,
$
\rho(\sigma_\ell)  
$ is either diagonal or skew symmetric.
In the latter case, the diagonal matrices form a subgroup of index 2 in $\rho(G_\ell)$.
But $\rho(G_\ell)$ is a pro-3 group, a contradiction. We conclude that $\rho(G_\ell)$ is abelian, and 
$$
\rho(\tau_\ell)^{\ell-1} = \begin{pmatrix}
    1 & \\ & 1
\end{pmatrix}.
$$
In particular, $\rho$ is finitely ramified at $\ell$ (whereas $\rho^{(\ell)}$ is infinitely ramified at $\ell$).
Let $k \geq 1$ be the integer such that $\ell \equiv 1 \bmod 3^k$ and 
$\ell \not \equiv 1 \bmod 3^{k+1}$.
Then $\rho(\tau_\ell)^{3^k} = I$.
This means that 
$$
\rho(\tau_\ell)^{3^k} = \begin{pmatrix}
    t^{3^k} & \\ & 1/t^{3^k}
\end{pmatrix}
= \begin{pmatrix}
    1 & \\ & 1 
\end{pmatrix}.
$$
That is, $t\in \overline{\Q}_3$ and $t^{3^k} = 1$, i.e. the minimal polynomial of $t$ over $\Q_3$ divides $X^{3^k}-1$. 
Thus, $t\in \Q_3(\zeta_{3^k})$ and $\Q_3(t) \subseteq \Q_3(\zeta_{3^k})$ where $\zeta_{3^k}$ is a primitive $3^k$th root of unity.
On the other hand, since $\rho$ is 2-dimensional and valued in $\operatorname{SL}(2,\Z_3),$
the eigenvalue $t$ is root of the quadratic polynomial
$$
p(\lambda) = \det (\rho(\tau_\ell) - \lambda I) \in \Z_3[\lambda]
$$
with coefficients in $\Z_3$. 
In total, $t$ lives in a quadratic extension of $\Q_3$ 
contained in $\Q_3(\zeta_{3^k})$.
Here, $\Gal(\Q_3(\zeta_{3^k})/\Q_3) \simeq (\Z/3^k\Z)^\times
\simeq \Z/3^{k-1}\Z \times \Z/2\Z$ admits $\Z/2\Z$ as a quotient in a unique way. In other words, $\Q_3(\zeta_{3^k})$ contains a unique subfield of degree 2 over $\Q_3$, namely $\Q_3(\zeta_3)$. 
We conclude that $t \in \Q_3(\zeta_3)$ and
\[
\rho(\tau_\ell)^3 = \begin{pmatrix}
    1 & \\ & 1
\end{pmatrix}.
\]
But then the ramification of $\rho$ at $\ell$ is exhausted already in $\overline{\rho}$. 
Recall that $\rho_n := \rho \bmod n$ for integers $n\geq 1$. 
Let $\Q(\rho_n)$ be the field fixed by $\ker \rho_n$ and 
recall that
$\Gal(\Q(\rho_{n+1})/\Q(\rho_n))$
is an $\F_3[\Gal(K/\Q)]$-submodule of 
$\Ad$ for all $n \geq 1$.
Suppose there exists some $n \geq 1$ for which the extension 
$\Q(\rho_{n+1})/\Q(\rho_{n})$ is not trivial,
and let $m \geq 1$ be the smallest such integer. 
Since $\Ad$ is irreducible, 
$$\Gal(\Q(\rho_{m+1})/\Q(\rho_{m})) \simeq \Ad$$ as $\F_3[\Gal(K/\Q)]$-modules.
Since $\rho(\tau_\ell)^3 = 1$, $\Q(\rho_{m+1})/\Q(\rho_m)$ must be unramified at all places not dividing 3.  
By the choice of $m \geq 1$, 
it must be the case that 
$\ker \rho_m = \ker \overline{\rho}$, i.e. 
$$\Q(\rho_m) = \Q(\overline{\rho}).$$
But then there is an extension $M/K$ 
unramified away from 3 with $\Gal(M/K)\simeq \Ad$, a contradiction.
We conclude that if $\rho(\tau_\ell)$ does not have 1 as an eigenvalue, then 
the field extension 
$\Q(\rho) / \Q(\overline{\rho})$
is trivial,
$
\ker \rho = \ker \overline{\rho},
$
and the image of $\rho$ is finite of the same order, 24, as $\overline{\rho}$. 
\end{proof}

\bibliographystyle{amsplain} 
\bibliography{references.bib}


\end{document}